\documentclass[a4paper, 11pt]{amsart}

\usepackage{amsfonts,amssymb,mathtools,amsmath}
\usepackage{textcomp, color}
\usepackage{hyperref}
\usepackage[capitalise]{cleveref}
\usepackage{wasysym}

\theoremstyle{plain}

\newtheorem*{thmA}{Theorem A}
\newtheorem*{thmB}{Theorem B}

\newtheorem{thm}{Theorem}[section]
\newtheorem{lem}[thm]{Lemma}
\newtheorem{pro}[thm]{Proposition}
\newtheorem{cor}[thm]{Corollary}

\theoremstyle{definition}

\newtheorem{dfn}[thm]{Definition}

\newtheorem{rmk}[thm]{Remark}

\newcommand{\Z}{\mathbb{Z}}

\newcommand{\C}{\mathbb{C}}
\renewcommand{\P}{\mathbb{P}}

\DeclareMathOperator{\Id}{Id}

\DeclareMathOperator{\Aut}{Aut}

\begin{document}

\title{Purely (non-) strongly real Beauville $p$-groups}

\author[\c{S}.\ G\"ul]{\c{S}\"ukran G\"ul}
\address{Department of Mathematics\\ University of the Basque Country UPV/EHU\\
48080 Bilbao, Spain}
\email{sukran.gul@ehu.eus}

\keywords{Beauville groups; strongly real Beauville $p$-groups; metacyclic $p$-groups; $p$-groups of class $2$\vspace{3pt}}

\thanks{ The author is supported by the Spanish Government, grant
MTM2017-86802-P, partly with FEDER funds, and by the Basque Government, grant IT974-16.}

\begin{abstract}
For every prime $p\geq5$, we give examples of Beauville $p$-groups whose Beauville structures are never strongly real.
This shows that there are purely non-strongly real nilpotent Beauville groups.
On the other hand, we determine infinitely many Beauville $2$-groups which are purely strongly real. 
This answers two questions formulated by Fairbairn in \cite{fai4}.
\end{abstract}

\maketitle

\section{Introduction}

A \emph{Beauville surface} (of unmixed type) is a compact complex surface isomorphic to a quotient $(C_1\times C_2)/G$ satisfying the following conditions:
\begin{enumerate}
\item 
$C_1$ and $C_2$ are algebraic curves of genus at least $2$,
and $G$ is a finite group acting freely on $C_1\times C_2$ by holomorphic transformations.
\item 
For $i=1,2$,  the group $G$ acts faithfully on each curve $C_i$ so that $C_i/G\cong \P_1(\C)$ and the covering map $C_i\rightarrow C_i/G$ is ramified over three points for $i=1,2$.
\end{enumerate}

If a group $G$ can be made to act on $C_1\times C_2$ as in (i) and (ii), then $G$ is said to be a \emph{Beauville group}.
The geometric definition of Beauville groups can be transformed into purely group-theoretical terms, see the remark after Definition 3.7 in \cite{cat2}.

\begin{dfn}
Let $G$ be a finite group. For a couple of elements $x, y \in G$, let
\[
\Sigma(x,y)
=
\bigcup_{g\in G} \,
\Big( \langle x \rangle^g \cup \langle y \rangle^g \cup \langle xy \rangle^g \Big).
\]
Then $G$ is called a \emph{Beauville group} if  $G$ is a $2$-generator group, and there exists a pair of generating sets $\{x_1,y_1\}$ and $\{x_2,y_2\}$ of $G$ such that 
$\Sigma(x_1,y_1) \cap \Sigma(x_2,y_2)=1$.
Then $ \{x_1,y_1\}$ and $\{x_2,y_2\}$ are said to form a \emph{Beauville structure\/} for $G$. 
\end{dfn}

Recall that a complex surface $S$ is called \emph{real} if there exists a biholomorphism  $\sigma \colon S\longrightarrow \overline{S}$ between $S$ and its complex conjugate surface $\overline{S}$ such that $\sigma^2=\Id$.
If we want Beauville surfaces to be real, then we work with the following sufficient condition on the Beauville structure, see page 38 in \cite{BCG}.

\begin{dfn}
Let $G$ be a Beauville group. We say that $G$ is \emph{strongly real\/} if there exist a Beauville structure $\{\{x_1,y_1\}, \{x_2,y_2\} \}$, an automorphism $\theta \in \Aut(G)$ and elements $g_i \in G$ such that 
\[
g_i \theta(x_i)g_i^{-1}=x_i^{-1} \ \ \text{and} \ \ g_i \theta(y_i)g_i^{-1}=y_i^{-1}
\]
for $i=1,2$.
Then the Beauville structure is called \emph{strongly real Beauville structure}.
\end{dfn}

We remark that a number of strongly real Beauville groups were given in \cite{fai, fai2, fai3, fai5, FGo, FJ, gul, gul2}.
Strongly real Beauville groups may have Beauville structures which are not strongly real. 
In this paper we are interested in the extreme cases given in the following definition.

\begin{dfn}
\label{dfn:purely}
A Beauville group $G$ is said to be a \emph{purely strongly real Beauville group} if all Beauville structures of $G$ are strongly real.
If none of the Beauville structures of $G$ are strongly real, then $G$ is said to be a \emph{purely non-strongly real Beauville group}.
\end{dfn}

Since for any abelian group the map $x \longmapsto x^{-1}$ is an automorphism, every abelian Beauville group is actually a purely strongly real Beauville group.
On the other hand, in \cite{fai4} Fairbairn showed that there are infinitely many purely non-strongly real Beauville groups by taking the direct product of the Mathieu group $M_{11}$ and any Beauville group of coprime order with $M_{11}$.
As these groups are not nilpotent, he asked if there exist any nilpotent purely non-strongly real Beauville groups.

In \cite{fai4} Fairbairn also constructed infinitely many non-abelian purely strongly real Beauville $p$-groups by considering one of the types of $2$-generator $p$-groups of nilpotency class $2$.
Namely, he showed that for $p\geq 5$ and $n\geq r\geq 1$, the group
\begin{equation}
\label{class2-case1}
G=\langle  x, y,z \mid  x^{p^n}=y^{p^n}=z^{p^r}=[x,z]=[y,z]=1, [x,y]=z\rangle
\end{equation}
is a purely strongly real Beauville group.
Since in his result $p\geq 5$, he asked about the existence of  purely strongly real Beauville $2$-groups and $3$-groups.

Our first goal in this paper is to give infinitely many purely non-strongly real Beauville $p$-groups for every prime $p\geq 5$.

\begin{thmA}
Let $G$ be a non-abelian metacyclic Beauville $p$-group or a Beauville $p$-group of nilpotency class $2$ which is not as in \eqref{class2-case1}. Then $G$ is a purely non-strongly real Beauville $p$-group.
\end{thmA}

The second main result in this paper  shows that there are infinitely many 
purely strongly real Beauville $2$-groups.

\begin{thmB}
For every $e\geq 2$, the group
\begin{multline*}
G=
\langle x,y,z,t,w \mid x^{2^e}=y^{2^e}=z^{2^{e-1}}=t^{2^{e-1}}=w^{2^{e-1}}=1,
\\
[y,x]=z, [z,x]=t, [z,y]=w \rangle
\end{multline*}
is a purely strongly real Beauville $2$-group.
 \end{thmB}

This paper is organized as follows:
in \cref{sec1}, we characterize non-abelian metacyclic Beauville $p$-groups and Beauville $p$-groups of nilpotency class $2$, and then
we also determine their automorphism groups.
In \cref{sec3}, we prove our main theorems.

\vspace{10pt}

\noindent
\textit{Notation.\/}
If $p$ is a prime and $G$ is a finite $p$-group, then $G^{p^i}=\langle g^{p^i} \mid g \in G \rangle$.
The exponent of $G$, denoted by $\exp G$, is the maximum of the orders of all elements of $G$.

\section{ Metacyclic Beauville $p$-groups and Beauville $p$-groups of nilpotency class $2$}
\label{sec1}

In this section, we will first give a classification of non-abelian metacyclic Beauville $p$-groups and Beauville $p$-groups of nilpotency class $2$. 
Then we will determine their automorphism groups.

\begin{thm}
\label{presentation-metacyclic}
Let $G$ be a non-abelian metacyclic Beauville $p$-group of exponent $p^e$.
Then $p\geq 5$ and $G$ has the following presentation:
\begin{equation}
\label{powerful beauville}
G=\langle a,b \mid a^{p^e}=b^{p^e}=1, [a,b]=a^{p^i} \rangle,
\end{equation}
where  $1 \leq i \leq e-1$.
\end{thm}

\begin{proof}
By Corollary 2.9 in \cite{FG}, $G$ is a Beauville group if and only if
$p\geq 5$ and $G$ is a semidirect product of two cyclic groups of order $p^e$.
This implies that $G$ has the presentation given in \eqref{powerful beauville}.
\end{proof}

\begin{rmk}
Recall that a finite $p$-group $G$ is \emph{powerful} if $G'\leq G^p$ for $p$ odd, or if $G′\leq G^4$ for $p=2$.
Then the groups given in \cref{presentation-metacyclic} are powerful since $p$ is odd.
On the other hand, by Exercise 2.13 in \cite{DdSMS}, a 2-generator powerful $p$-group is metacyclic.
As a consequence, non-abelian metacyclic Beauville $p$-groups are exactly the same as non-abelian powerful Beauville $p$-groups.  
\end{rmk}

We next deal with  Beauville $p$-groups of nilpotency class $2$.

\begin{pro}
\label{p=2}
There is no Beauville $2$-group of nilpotency class $2$.
\end{pro}

\begin{proof}
Let $G$ be a $2$-generator $2$-group of nilpotency class $2$ of exponent $2^e$. 
Since the class is $2$, $G'$ is cyclic and hence it consists of commutators.
For every $x, y \in G$ we have
\[
1=(xy)^{2^e}=x^{2^e}y^{2^e}[y,x]^{\binom{2^e}{2}}=[y,x]^{\binom{2^e}{2}},
\]
and thus $\exp G'\leq 2^{e-1}$. 

On the other hand, let us see that $\exp \Phi(G)\leq 2^{e-1}$. 
Any element in $\Phi(G)$ can be written in the form $x^2c$ for some $x\in G$ and $c\in G'$. Since $G'\leq Z(G)$  and $\exp G'\leq 2^{e-1}$, we have
$(x^2c)^{2^{e-1}}=x^{2^e}c^{2^{e-1}}=1$.

Thus, as $\exp G=2^e$, there exists an element $a \in G\smallsetminus \Phi(G)$ which is of order $2^e$. 
Consider the maximal subgroup $M$ containing $a$. 
We claim that all elements in $M\smallsetminus \Phi(G)$ are of order $2^e$ and their $2^{e-1}$st powers coincide with $a^{2^{e-1}}$. 
Every element in $M\smallsetminus \Phi(G)$ is of the form $ab$ where $b\in \Phi(G)$.
If we write $b=x^2c$ for some $x\in G$ and $c\in G'$, then 
\[
(ax^2c)^{2^{e-1}}=a^{2^{e-1}}(x^2c)^{2^{e-1}}[x^2c, a]^{\binom{2^{e-1}}{2}}
=a^{2^{e-1}}[x,a]^{2\binom{2^{e-1}}{2}}=a^{2^{e-1}}.
\]
Since there are only three maximal subgroups in $G$,  for any generating set 
$\{x_1, y_1\}$ of $G$ one of the elements in 
$\{x_1, y_1, x_1y_1\}$ falls in the maximal subgroup $M$, and hence $a^{2^{e-1}} \in \Sigma (x_1, y_1)$. 
Thus, $G$ cannot be a Beauville group. 
\end{proof}

The following result, which gives the complete classification of $2$-generator $p$-groups of nilpotency class $2$, is  Theorem 1.1 in \cite{AMM}.

\begin{thm}
\label{presentation general}
Every $2$-generator $p$-group of order $p^n$ of nilpotency class $2$ corresponds to an ordered $5$-tuple of integers, $(\alpha, \beta, \gamma; \rho, \sigma)$, such that
\begin{enumerate}
\item
$\alpha \geq \beta \geq \gamma \geq 1$,
\item 
$\alpha+\beta+\gamma=n$,
\item
$0\leq \rho, \sigma \leq \gamma$,
\end{enumerate}
where $(\alpha, \beta, \gamma; \rho, \sigma)$ corresponds to the group presented by 
\[
G=\langle x,y \mid 
[x,y]^{p^{\gamma}}=[x,y,x]=[x,y,y]=1, \ x^{p^{\alpha}}=[x, y]^{p^{\rho}}, \ y^{p^{\beta}}=[x, y]^{p^{\sigma}}   \rangle.
\]
\end{thm}

By using this result, we will characterize all Beauville $p$-groups of nilpotency class $2$.
Before that, we recall the definition of regular $p$-groups.

\begin{dfn}
A finite $p$-group $G$ is called \emph{regular} if for every $x, y \in G$,
$(xy)^p\equiv x^py^p \pmod{H^p}$ where $H=\langle x,y\rangle'$.
\end{dfn}

\begin{thm}
Let $G$ be a  Beauville $p$-group of nilpotency class $2$ and of exponent $p^e$.
Then $p\geq 5$ and  $G$ has the following presentation:
\begin{equation}
\label{class 2 beauville}
G=\langle a,b \mid a^{p^e}=[b,a]^{p^j}=[b,a,a]=[b,a,b]=1, b^{p^i}=[b,a]^{p^k} \rangle,
\end{equation}
where $0 \leq k \leq  j \leq i \leq e$ and $e=i+j-k$.
\end{thm}

\begin{proof}
Let $G=\langle x, y \rangle$ have a presentation as in the previous theorem. 
By \cref{p=2}, $p$ is odd. 
Since $G$ is of class $2<p$, this implies that $G$ is a regular $p$-group. 
Then by Corollary 2.6 in \cite{FG}, we have $p\geq 5$ and $|G^{p^{e-1}}|\geq p^2$. 
Also according to Proposition 2.7 in \cite{FG},  the condition  $|G^{p^{e-1}}|\geq p^2$ holds if and only if $\langle x^{p^{e-1}}\rangle$ and $\langle y^{p^{e-1}}\rangle$ are distinct and non-trivial. 
Therefore, $\langle x \rangle \cap \langle y \rangle=1$ and $o(x)=o(y)=p^e$.

By the condition $\langle x \rangle \cap \langle y \rangle=1$, in the presentation of $G$ in \cref{presentation general}, we have either $\sigma=\gamma$ or $\rho=\gamma$.
On the other hand, since $p^e=o(x)=p^{\alpha+\gamma-\rho}$ and $p^e=o(y)=p^{\beta+\gamma-\sigma}$, the condition $\alpha\geq \beta$
implies that $\sigma\leq\rho$.
As a consequence, we necessarily have $\rho=\gamma$.
Hence $G$ has the following presentation:
\begin{equation*}
\label{first presentation-class 2}
\langle x, y \mid x^{p^e}=[x,y]^{p^j}=[x,y,x]=[x,y,y]=1, y^{p^i}=[x,y]^{p^k}\rangle,
\end{equation*}
where $0 \leq k \leq  j \leq i \leq e$ and $e=i+j-k$.

Now consider the group 
\[
L=\langle a, b \mid a^{p^e}=[b,a]^{p^j}=[b,a,a]=[b,a,b]=1, b^{p^i}=[b,a]^{p^k}\rangle,
\]
where $0 \leq k \leq  j \leq i \leq e$ and $e=i+j-k$.
Then one can easily check that the map $\theta \colon L \longrightarrow G$ defined by $\theta(a)=x^{-1}$ and $\theta(b)=y$ is an isomorphism.
This completes the proof.
\end{proof}

\begin{rmk}
\label{class 2-types}
Note that we have three main types of Beauville $p$-groups of nilpotency class $2$.
The first type is when $k=0$ in \eqref{class 2 beauville}. 
In this case we have $[b,a]=b^{p^i}$, and hence the group is powerful.
If $k=j$ then \eqref{class 2 beauville} coincides with the presentation given in  \eqref{class2-case1}.
In the last type we have $0< k < j$.
\end{rmk}

\vspace{10pt}

We continue this section by giving the automorphism groups of non-abelian metacyclic Beauville $p$-groups and  Beauville $p$-groups of nilpotency class $2$ which are not as in \eqref{class2-case1}.
To this purpose, we refer to the paper \cite{men} by Menegazzo.

\begin{pro}
\label{automorphism-powerful}
Let $p$ be an odd prime and let $G$ be a metacyclic $p$-group with the following presentation:
\begin{equation*}
G=\langle a,b \mid a^{p^e}=b^{p^e}=1, [a,b]=a^{p^i} \rangle,
\end{equation*}
where  $1 \leq i \leq e-1$.
Then a map $\theta$ defined on the generators $a$ and $b$ extends to an automorphism of $G$ if and only if
\[
\theta(a)= b^{mp^{e-i}}a^n \ \ \ \
\text{and} \ \ \ \
\theta(b)=b^{1+rp^{e-i}}a^{s},
\]
where $1\leq m, r\leq p^{i}$ and $1\leq n, s\leq p^e$ with $p\nmid n$. 
\end{pro}

\begin{proof}
By the result of Menegazzo in \cite[page 82]{men}, a map $\theta$ defined on the generators $a$ and $b$ extends to an automorphism of $G$ if and only if
\[
\theta(a)= b^{\alpha}a^{\beta} \ \ \ \
\text{and} \ \ \ \
\theta(b)= b^{1+\lambda p^{e-i}}(b^{\alpha}a^{\beta})^{\mu},
\]
where $p^{e-i} \mid \alpha$ and $p\nmid \beta$.
As one can easily check that $b^{p^{e-i}}\in Z(G)$, we have $(b^{\alpha}a^{\beta})^{\mu}=b^{\alpha\mu}a^{\beta\mu}$.
Then by writing $\alpha=mp^{e-i}$, $\beta=n$, $\lambda+m\mu=r$ and $\beta\mu=s$, we get the result.
 \end{proof}

We next consider Beauville $p$-groups of nilpotency class $2$ which are  not as in \eqref{class2-case1}.
By \cref{class 2-types}, there is only one case that we have to deal with since in the other case the groups are powerful.

\begin{pro}
\label{automorphism-class 2}
 Let $p$ be an odd prime and let $G$ be a $p$-group given by the
 following presentation:
\[
G=\langle a,b \mid a^{p^e}=[b,a]^{p^j}=[b,a,a]=[b,a,b]=1, b^{p^i}=[b,a]^{p^k} \rangle,
\]
where $0<k<j\leq i\leq e$ and $e= i+j-k$.
Then a map $\theta$ defined on the generators $a$ and $b$ extends to an automorphism of $G$ if and only if
\[
\theta(a)=a^{1+mp^{e-i}}b^nc_a \ \ \ \
\text{and} \ \ \ \
\theta(b)= a^{rp^{e-i}}b^sc_b,
\]
where  $c_a,c_b \in G'$ and $1\leq m,n,r,s\leq p^i$ with $p \nmid s$.
\end{pro}

\begin{proof}
By the result of Menegazzo in \cite[page 85]{men}, a map $\theta$ defined on the generators $a$ and $b$ extends to an automorphism of $G$ if and only if
\[
\theta(a)= a^{1+\lambda p^{e-i}}c \ \ \ \
\text{and} \ \ \ \
\theta(b)= a^{\alpha}b^{\beta}[b,a]^{\mu},
\]
where $p^{e-i} \mid \alpha$, $p\nmid \beta$ and $c\in \langle a^{\alpha}b^{\beta}\rangle G'$.

Now if we write $\alpha=rp^{e-i}$, and $\beta=s$, then 
\[
\theta(b)= a^{rp^{e-i}}b^sc_b,
\]
where $1\leq r,s \leq p^i$ with $p\nmid s$, and $c_b\in G'$.

On the other hand, observe that  all elements in $ \langle a^{\alpha}b^{\beta}\rangle G'$ are of the form
\[
a^{r\ell p^{e-i}}b^{s\ell}c_a,
\]
where $1\leq r\ell, s\ell \leq p^i$ and $c_a\in G'$.
As a consequence, 
\[
\theta(a)=a^{1+mp^{e-i}}b^nc_a,
\]
where $1\leq m, n \leq p^i$, and $c_a\in G'$, as desired.
\end{proof}

\section{Proof of the main theorems}
\label{sec3}

This section is devoted to proving our main theorems.
We will first deal with Theorem A.
Indeed, Theorem A immediately follows from the following result.

\begin{thm}
Let $p$ be an odd prime and let $G$ be a $2$-generator $p$-group of exponent $p^e$ with one of the following presentations:
\begin{enumerate}
\item 
$G=\langle a,b \mid a^{p^e}=b^{p^e}=1, [a,b]=a^{p^i}\rangle$,
where $1 \leq i \leq e-1$,
\item 
$
G=\langle a,b \mid a^{p^e}=[b,a]^{p^j}=[b,a,a]=[b,a,b]=1, b^{p^i}=[b,a]^{p^k} \rangle,
$
where $0<k<j\leq i\leq e$ and $e= i+j-k$.
\end{enumerate}
Let $\{x , y\}$ be an arbitrary generating set of $G$.
Then there do not exist $\theta \in \Aut(G)$ and $g\in G$ such that
\[
\theta(x)=(x^{-1})^{g} \ \
\text{and}
\ \
\theta(y)=(y^{-1})^{g}.
\]
\end{thm}

\begin{proof}
Suppose, by way of contradiction, that there exist $\theta \in \Aut(G)$ and $g \in G$ such that $\theta(x)=(x^{-1})^{g}$ and $\theta(y)=(y^{-1})^{g}$. 
We next write
$x=a^tb^uc$ and $y=a^vb^wd$ where $c, d \in \Phi(G)$ and $t, u, v, w \in \Z$.

Now we have
\begin{equation}
\label{condition failed}
\theta(x)=\theta(a^tb^uc)=\theta(a)^t\theta(b)^u\theta(c)=(c^{-1}b^{-u}a^{-t})^g.
\end{equation}
We first assume that we are in  case (i). 
According to \cref{automorphism-powerful}, we have $\theta(x)\equiv a^{tn+us}b^u \pmod{\Phi(G)}$ for some $n, s \in \Z$.
Then \eqref{condition failed} implies that
\[
a^{tn+us}b^u \equiv a^{-t}b^{-u} \pmod{\Phi(G)},
\]
that is, $a^{t(n+1)+us} \equiv b^{-2u} \pmod{\Phi(G)}$. 
Since $p$ is odd, this implies that $p \mid u$.
By using the same arguments for $y$, we also get $p \mid w$.
Thus, $G=\langle x, y\rangle \leq \langle a, \Phi(G)\rangle$, a contradiction.

Similarly, if we are in case (ii), then by using \cref{automorphism-class 2}, we get $p\mid t$ and $p\mid v$. 
Consequently, in this case we have $G=\langle x, y\rangle \leq \langle b, \Phi(G) \rangle$, which is again a contradiction.
\end{proof}

The next result, which answers Question 18 in \cite{fai4}, is a straightforward consequence of Theorem A.

\begin{cor}
There exist infinitely many nilpotent purely non-strongly real Beauville groups.
\end{cor}

In the remaining part of this section, we will prove Theorem B.
Before we proceed, we will need to introduce a general lemma.

\begin{lem}
\label{reducing case distinctions}
Let $G$ be a group and let $x,y \in G$. 
If $\theta\in \Aut(G)$, then the following are equivalent.
\begin{enumerate}
\item 
There exists $g\in G$ such that $\theta(x)=(x^{-1})^g$ and $\theta(y)=(y^{-1})^g$.
\item 
For $a=xy$ and $b\in \{x,y\}$, there exists $h\in G$ such that
$\theta(a)=(a^{-1})^h$ and $\theta(b)=(b^{-1})^h$.
\end{enumerate}
\end{lem}

\begin{proof}
We first assume that (i) holds.
Since $\theta \in \Aut(G)$, we have 
\[
\theta(a)=\theta(x)\theta(y)=(x^{-1}y^{-1})^g.
\]
Note that $a^{-1}=(x^{-1}y^{-1})^{y}=(x^{-1}y^{-1})^{x^{-1}}$.
Hence $\theta(a)=(a^{-1})^{y^{-1}g}=(a^{-1})^{xg}$.
Then (ii) follows by the fact that $\theta(x)=(x^{-1})^{xg}$ and $\theta(y)=(y^{-1})^{y^{-1}g}$.

Conversely, assume that (ii) holds.
Since we can apply a similar argument we assume that $b=x$.
Then the equalities $\theta(a)=\theta(xy)=(y^{-1}x^{-1})^h$ and $\theta(x)=(x^{-1})^h$ imply that $\theta(y)=(y^{-1})^{x^{-1}h}$.
Thus, we take $g=x^{-1}h$.
\end{proof}

For $e\geq2$, let $T$ be the triangle group defined by the presentation
\[
T=\langle a, b \mid a^{2^e}=b^{2^e}=(ab)^{2^{e}}=1\rangle.
\]
Then by Theorem 2.1(iii) in \cite{gul2}, the group $G$ given by the following presentation
\begin{multline}
\label{2-group-presentation}
G=
\langle x,y,z,t,w \mid x^{2^e}=y^{2^e}=z^{2^{e-1}}=t^{2^{e-1}}=w^{2^{e-1}}=1,
\\
[y,x]=z, [z,x]=t, [z,y]=w \rangle
\end{multline}
is isomorphic to $T/\gamma_4(T)$.

Note that the map $\phi \colon T \longrightarrow T$ defined by $\phi(a)=a^{-1}$ and $\phi(b)=b^{-1}$ is an automorphism. 
Since $G\cong T/\gamma_4(T)$, the map $\phi$ induces an automorphism $\theta \colon G\longrightarrow G$ defined by $\theta(x)=x^{-1}$ and $\theta(y)=y^{-1}$.
Then observe that we have $\theta(z)=zt^{-1}w^{-1}$, \ $\theta(t)=t^{-1}$ and $\theta(w)=w^{-1}$.

\vspace{10pt}

In the calculations in \cref{candidates for g } and \cref{second main theorem}, we will use the following identity.
Let $G$ be a group as in \eqref{2-group-presentation}. 
Then since $G$ is of class $3$, we have
\[
[y^m, x^n]=[y^m, x]^n[y^m, x, x]^{\binom{n}{2}},
\]
where $[y^m,x]=z^mw^{\binom{m}{2}}$. 
As a consequence,
\[
[y^m, x^n]=z^{mn}t^{m\binom{n}{2}}w^{n\binom{m}{2}}.
\]

The following technical lemma will be needed for the proof of \cref{second main theorem}.

\begin{lem}
\label{candidates for g }
Let $G$ be a group as in \eqref{2-group-presentation}, and let $\theta \in \Aut(G)$ be the automorphism of $G$ defined by $\theta(x)=x^{-1}$ and $\theta(y)=y^{-1}$.
Then the following hold:
\begin{enumerate}
\item
If $a=x^iy^jz^k \in G$ with $1\leq i,j \leq 2^e$, $1\leq k \leq 2^{e-1}$, and $i$ is odd, then
\[
a\theta(a)=[a^{-1}, y^{2kn-j}],
\]
where $1\leq n <2^e$ and $i\cdot n\equiv 1 \pmod{2^e}$.
\item 
If $b=y^jx^iz^k \in G$ with $1\leq i,j \leq 2^e$, $1\leq k \leq 2^{e-1}$, and $j$ is odd, then
\[
b\theta(b)=[b^{-1}, x^{-2km-i}],
\]
where $1\leq m <2^e$ and $j \cdot m\equiv 1 \pmod{2^e}$.
\end{enumerate}
\end{lem}

\begin{proof}
We will first prove (i).
We have
\[
[a^{-1}, y^{2kn-j}]=[a^{-1}, y^{-j}][a^{-1}, y^{2kn}]^{y^{-j}}.
\]
By keeping in mind that $i \cdot n\equiv 1 \pmod{2^e}$, we have
\[
[a^{-1}, y^{2kn}]=z^{2k}t^{-ik-k}w^{-k},
\]
and hence
\[
[a^{-1}, y^{2kn-j}]=[a^{-1}, y^{-j}]z^{2k}t^{-ik-k}w^{-2jk-k}.
\]
On the other hand, 
\begin{align*}
a\theta(a)&= ax^{-i}y^{-j}z^kt^{-k}w^{-k}\\
&= ay^{-j}x^{-i}[z^{-k}y^{-j}x^{-i},y^{-j}]z^kt^{-k}w^{-jk-k}\\
&=ay^{-j}x^{-i}[a^{-1}, y^{-j}]z^kt^{-k}w^{-jk-k}\\
&=[a^{-1}, y^{-j}]z^{2k}t^{-ik-k}w^{-2jk-k}.
\end{align*}
Consequently, (i) holds.

Notice that (ii) follows from (i). 
In item (i), we interchange the roles of $x$ and $y$. 
Then since $z=[y, x]$, we have to replace $z$ with $z^{-1}$ and hence we get (ii) from (i).
\end{proof}

We are now ready to prove Theorem B, which answers Question 13 in \cite{fai4}.

\begin{thm}
\label{second main theorem}
For every $e\geq 2$, the group
\begin{multline*}
G=
\langle x,y,z,t,w \mid x^{2^e}=y^{2^e}=z^{2^{e-1}}=t^{2^{e-1}}=w^{2^{e-1}}=1,
\\
[y,x]=z, [z,x]=t, [z,y]=w \rangle
\end{multline*}
is a purely strongly real Beauville $2$-group.
\end{thm}

\begin{proof}
By the proof of Theorem 2.1(iii) in \cite{gul2}, $G$ can be constructed as the semidirect product of $\langle y\rangle  \ltimes A$ by $\langle x\rangle\cong C_{2^e}$, where $\langle y\rangle\cong C_{2^e}$ and $A=\langle z \rangle \times \langle t\rangle \times \langle w \rangle \cong C_{2^{e-1}}\times C_{2^{e-1}}\times C_{2^{e-1}}$.
Hence $\exp G'=2^{e-1}$.
Again by Theorem 2.1(iii)  in \cite{gul2}, we know that $\exp G=2^e$.
This implies that any two elements which generate $G$ satisfy the relations given in \eqref{2-group-presentation}.

Let $\{x_1, y_1\}$ and $\{ x_2, y_2\}$ form a Beauville sturcture for $G$.
Then by the previous paragraph, without loss of generality we may assume that $x_1=x$ and $y_1=y$.
Since the map $\theta \colon G\longrightarrow G$ defined by $\theta(x)=x^{-1}$ and $\theta(y)=y^{-1}$ is an automorphism, in order to complete the proof we need to show that there exists $h\in G$ such that $\theta(x_2)=(x_2^{-1})^h$ and $\theta(y_2)=(y_2^{-1})^h$.

Note that $G$ has $3$ maximal subgroups.
Hence, there are $a, b \in \{x_2, y_2, x_2y_2\}$ such that $a\in \langle x, \Phi(G)$ and $b\in \langle y, \Phi(G)\rangle$.
Then according to  \cref{reducing case distinctions}, it is enough to show that $\theta(a)=(a^{-1})^g$ and $\theta(b)=(b^{-1})^g$ for some $g\in G$.

Indeed, we will prove more generally that for any 
$u \in \langle x, \Phi(G)\rangle \smallsetminus \Phi(G)$ and $v \in \langle y, \Phi(G)\rangle \smallsetminus \Phi(G)$, there exists some $g \in G$ such that $\theta(u)=(u^{-1})^g$ and $\theta(v)=(v^{-1})^g$.
Since the restriction of $\theta$ to $Z(G)=\langle t, w\rangle$ coincides with the inversion map, we can assume that
\begin{equation}
\label{u and v}
u=x^{1+2i_1}y^{2j_1}z^{k_1} \ \ \
\text{and} \ \ \
v=y^{1+2j_2}x^{2i_2}z^{k_2},
\end{equation}
where $1\leq i_{s}, j_{s}, k_{s} \leq 2^{e-1}$ for $s=1, 2$.

Now according to \cref{candidates for g }, we have
\[
u\theta(u)=[u^{-1}, y^{2k_1n-2j_1}] \ \ \
\text{and} \ \ \
v\theta(v)=[v^{-1}, x^{-2k_2m-2i_2}],
\]
where 
\begin{equation}
\label{inverses}
(1+2i_1)\cdot n \equiv 1 \pmod{2^e} \ \ \
\text{and} \ \ \
(1+2j_2)\cdot m \equiv 1 \pmod{2^e}.
\end{equation}
Then the equalities  $\theta(u)=(u^{-1})^g$ and $\theta(v)=(v^{-1})^g$ hold if and only if 
\[
[u^{-1}, y^{2k_1n-2j_1}]=[u^{-1},g] \ \ \
\text{and} \ \ \
[v^{-1}, x^{-2k_2m-2i_2}]=[v^{-1},g].
\]
Thus, such an element $g\in G$ exists if and only if
 \begin{equation}
 \label{elements in centralizer}
 gy^{2j_1-2k_1n} \in C_G(u) \ \ \
 \text{and}\ \ \
 gx^{2k_2m+2i_2} \in C_G(v).
 \end{equation}
 We will next show that $C_G(u)=\langle u \rangle Z(G)$ and $C_G(v)=\langle v \rangle Z(G)$.
 Clearly, $\langle u \rangle Z(G) \leq  C_G(u)$. 
 Conversely, let $c \in C_G(u)$ and write $c \equiv u^{\alpha}v^{\beta}[v,u]^{\mu} \pmod{Z(G)}$ where $\alpha, \beta, \mu \in \Z$.
 Then 
 \begin{align*}
 1=[u, c]&=[u, u^{\alpha}v^{\beta}[v,u]^{\mu}]\\
 &=[u, [v,u]^{\mu}][u, v^{\beta}]\\
 &=[u,v]^{\beta}[u,v,v]^{\binom{\beta}{2}}[u,v,u]^{\mu}
\end{align*}
if and only if $2^{e-1} \mid \mu$ and $2^e \mid \beta$.
Thus, $c\equiv u^{\alpha} \pmod{Z(G)}$ and hence $C_G(u)=\langle u \rangle Z(G)$, as desired.
Similarly, we have $C_G(v)=\langle v \rangle Z(G)$.

Thus by \eqref{elements in centralizer}, we need to check whether there exist $1\leq R, S\leq 2^{e}$ such that
 \begin{equation}
 \label{congruence of g}
u^Sy^{2k_1n-2j_1}\equiv v^Rx^{-2k_2m-2i_2}\pmod{Z(G)}.
 \end{equation}
 Now we will replace $u$ and $v$ with their values in \eqref{u and v}.
 Then by taking into account that modulo $Z(G)$ the group $G$ has nilpotency class $2$ and by \eqref{inverses}
 we obtain that \eqref{congruence of g} is equivalent to the following:
 \begin{multline}
 x^{(1+2i_1)S} y^{2j_1(S-1)+2k_1n}  z^{(1+2i_1)j_1S(S-1)+k_1S} \\
 \equiv
 x^{2i_2(R-1)-2k_2m}y^{(1+2j_2)R} z^{(1+2j_2)i_2R(R-1)-k_2R}
 \pmod{Z(G)}.
 \end{multline}
 As a consequence, we have the following three equations:
 \begin{equation}
 \label{powers of x}
 (1+2i_1)S
 \equiv 
 2i_2(R-1)-2k_2m
 \pmod{2^e},
\end{equation}
 \begin{equation}
\label{powers of y}
2j_1(S-1)+2k_1n
\equiv
(1+2j_2)R
\pmod{2^e},
\end{equation}
and 
\begin{equation}
\label{powers of z}
(1+2i_1)j_1S(S-1)+k_1S
\equiv
(1+2j_2)i_2R(R-1)-k_2R
\pmod{2^{e-1}}.
\end{equation}
 By taking into account \eqref{inverses}, if we multiply equations \eqref{powers of x} and \eqref{powers of y} side by side and then divide by $2$,  we get \eqref{powers of z}.
 
Thus, it only remains to check the existence of values of $R$ and $S$ satisfying equations \eqref{powers of x} and \eqref{powers of y}. 
Now we multiply both sides of \eqref{powers of x} by $-2j_1$, and both sides of \eqref{powers of y} by $1+2i_1$.
Then we sum up these two equations, so that the only unknown will be $R$, which has an odd coefficient. 
Hence, we can get the value of $R$. 
We next substitute $R$ in \eqref{powers of x}, and consequently we obtain $S$ since its coefficient is odd.

As a result, \eqref{congruence of g} has a solution in $R$ and $S$.
Therefore, we can take 
\[
g=u^Sy^{2k_1n-2j_1}=v^Rx^{-2k_2m-2i_2}d
\]
for some $d\in Z(G)$.
This completes the proof.
\end{proof}

\end{document}